\documentclass{amsart}

\usepackage[colorlinks=true, urlcolor=blue,bookmarks=true,bookmarksopen=true,citecolor=blue,hypertex]{hyperref}
\usepackage{graphicx}
\usepackage{enumerate}

\usepackage{epsfig}
\usepackage{amscd}
\usepackage{amssymb}
\usepackage{amsxtra}
\usepackage{amsmath}
\usepackage{enumerate}
\usepackage{mathrsfs}

\usepackage{color}

\usepackage{amsmath,amsfonts,amssymb}
\usepackage{verbatim}
\usepackage[hmargin = 4cm,vmargin = 2.5cm]{geometry}
\usepackage{epsfig}
\usepackage{amsthm}

\theoremstyle{plain}

\newtheorem{thm}{Theorem}

\newtheorem{lem}[thm]{Lemma}

\theoremstyle{definition}

\theoremstyle{remark}
\newtheorem{rem}[thm]{Remark}

\theoremstyle{plain}

\newcommand{\Id}{{{\mathchoice {\rm 1\mskip-4mu l} {\rm 1\mskip-4mu l}
      {\rm 1\mskip-4.5mu l} {\rm 1\mskip-5mu l}}}}
\renewcommand\ker{\operatorname{ker}}

\def\ZZ{\mathbb{Z}}

\def\RR{\mathbb{R}}

\def\dd{\text{d}}
\def\Ham{Ham}
\def\Cal{Cal}

\begin{document}

\bibliographystyle{alphanum}

\title{Hofer's norm and disk translations in an annulus}
\date{\today} 
\author{Michael Khanevsky}
\address{Michael Khanevsky, School of Mathematics, 
Institute for Advanced Study, Einstein Drive, Princeton, NJ 08540, USA}
\email{khanev@math.ias.edu}

\begin{abstract}
Let $\mathcal{D}$ be a non-displaceable disk in an annulus $\mathbb{A}$. Suppose that $\phi$ is a compactly supported 
Hamiltonian such that $\phi (\mathcal{D}) = \mathcal{D}$ with translation number $n$.
We show that Hofer's norm $\| \phi \|$ is bounded from below by
$\kappa \cdot |n|$ for a certain constant $\kappa$ which depends on $\mathcal{D}$ but not on $n$. 
We also give example of such Hamiltonian which is more efficient than the obvious
rotation of $\mathcal{D}$. This answers question 5 from the list proposed 
by Fr\'{e}d\'{e}ric Le Roux in ~\cite{LR:6-questions}.
\end{abstract}

\maketitle

\section{Introduction and results}

Let $\mathbb{A} = S^1 \times (0, 1)$ be an annulus equipped with the 
standard symplectic form $\omega$. We use the convention $S^1 = \RR / \ZZ$, so
$\int_\mathbb{A} \omega = 1$. 
Let $\mathcal{D} \subset \mathbb{A}$ be a disk
with $A = Area (\mathcal{D}) = \int_\mathcal{D} \omega > 1/2$, $L = \partial \mathcal{D}$. 
Denote by $S = \{ \phi \in Ham (\mathbb{A}) \, | \, \phi (\mathcal{D}) = \mathcal{D}\}$ the stabilizer of $\mathcal{D}$ in 
the group $\Ham(\mathbb{A})$ of compactly supported Hamiltonian diffeomorphisms.

For a $\phi \in S$ we define the \emph{translation number} $\tau_\mathcal{D} (\phi)$ in the following way.
Roughly speaking, $\tau_\mathcal{D}$ is the number of loops around the $S^1$ coordinate that $\mathcal{D}$ does under an isotopy 
from the identity to $\phi$. For the formal definition, let $\widetilde{\phi}$ be the lift of $\phi$
to the universal cover $\widetilde{\mathbb{A}} \simeq \RR \times (0, 1)$ which restricts 
to $\Id$ near the boundary. Pick $p \in \mathcal{D}$ and 
pick a lift $\tilde{p}$ of $p$ to $\widetilde{\mathbb{A}}$.
Then $$\tau_\mathcal{D} (\phi) = \lim_{n \to \infty} \frac{\pi_\RR ({\widetilde{\phi}}^n (\tilde{p}))}{n}.$$
It is easy to see that $\tau_\mathcal{D}$ does not depend on the choices made. If $p \in \mathcal{D}$ is a fixed 
point of $\phi$ then $\tau_\mathcal{D} (\phi)$ agrees with the usual notion of translation number
for $\phi$ and $p$. We denote $S_n = \{\phi \in S \, | \, \tau_\mathcal{D} (\phi) = n\}$.

Denote by $\| \cdot \|$ the Hofer norm on the group $\Ham(\mathbb{A})$: 
\[
	\| \phi \| = \inf \int_0^1 \max_{p \in \mathbb{A}} H (p, t) - 
																								\min_{p \in \mathbb{A}} H (p, t) \mathrm{d}t ,
\]
where the infimum goes over all compactly supported Hamiltonians $H: \mathbb{A} \times [0, 1] \to \RR$ 
such that $\phi$ is the time-1 map of the corresponding flow.

We prove the following:
\begin{thm} \label{T:bound}
\[
	 \frac{2A-1}{2} \cdot |n| \leq \inf_{\phi \in S_n}{\|\phi\|} < (2 A - 1) \cdot |n| + 1.
\]
\end{thm}

This answers question 5 from the list provided by F. Le Roux in ~\cite{LR:6-questions}.
The author provides there the following background for this question.
Let $A$ be the annulus $S^1 \times (a, b) \subset \mathbb{A}$ ($0 < a < b < 1$) and 
$D' \subset \mathbb{A}$ be a displaceable disk. Consider Hamiltonians $\phi_A, \phi_{D'}$
which fix $A$ (resp., $D'$) pointwise with translation number $n$. Then, according to 
~\cite{LR:6-questions}, the energy-capacity inequality in the universal cover 
$\widetilde{\mathbb{A}}$ of $\mathbb{A}$ implies $\|\phi_A\| \geq |n| \cdot Area (A)$.
On the other hand, $\| \phi_{D'} \|$ could be less than $1$ regardless of $n$.
The case of a non-displaceable disk $\mathcal{D}$ may be considered as an ``intermediate''
between $A$ and $D'$. F. Le Roux asked whether the translations of $\mathcal{D}$ 
behave similarly to those of one of the sets above. 
Theorem~\ref{T:bound} shows that Hofer's norm of translations of $\mathcal{D}$
grows linearly with respect to the translation number (in a similar way to the annulus $A$) 
but the coefficient is strictly less than $Area(\mathcal{D})$, in contrast to the case of an annulus.

\medskip

The rest of this paper is organized as follows.
In Section~\ref{S:qm} we extend the translation number $\tau_\mathcal{D}$ 
to a quasimorphism $\rho : \Ham(\mathbb{A}) \to \RR$. $\rho$ is Hofer-Lipschitz, hence gives 
a lower bound on Hofer's distance. This implies the left-hand side inequality in 
Theorem~\ref{T:bound}. For the right-hand side inequality we describe 
an explicit Hamiltonian flow whose time-$1$ map $\psi_n$ belongs to $S_n$ and whose 
length is bounded by $(2 A - 1) \cdot |n| + 1$. The details are carried out in 
Section~\ref{S:diff}.

\medskip

\emph{Acknowledgements:}
The author would like to thank F. Le Roux for his useful comments on this work.
Special thanks to M. Maydanskiy for a crash course on mapping class group.
The author thanks Eidgen\"{o}ssische Technische Hochschule in Z\"{u}rich for its hospitality during 
this research. This material is based upon work partially supported by the National Science
Foundation under agreement No. DMS-0635607. Any opinions, findings and conclusions or 
recommendations expressed in this material are those of the author and do not necessarily reflect the views
of the National Science Foundation.

\section{Lower bound}\label{S:qm}

Let $G$ be a group. A function $r : G \to \RR$ is called a \emph{quasimorphism} if there exists 
a constant $R$ such that $|r(fg) - r(f) - r(g)| \leq R$ for all $f, g \in G$. 
$R$ is called the \emph{defect} of $r$. The quasimorphism $r$ 
is \emph{homogeneous} if it satisfies $r(g^m) = mr(g)$ for all $g \in G$ and $m \in \ZZ$.
Any homogeneous quasimorphism satisfies $r(fg) = r(f) + r(g)$ for commuting elements $f, g$.

Pick a function $\widehat{H} : \mathbb{A} \to \RR$ with compact support such that $\widehat{H}(\theta, h) = h$
away from a small neighborhood of $\partial \mathbb{A}$. Denote by $\widehat{\Phi}$ the time-$1$ map of the 
Hamiltonian flow generated by $\widehat{H}$. It is easy to see that $\widehat{\Phi} \in S$ with translation number 
$\tau_\mathcal{D} (\widehat{\Phi}) = 1$. Note that $S = \bigcup_{n \in \ZZ} S_n = \bigcup_{n \in \ZZ} \widehat{\Phi}^n S_0$.

In what follows we construct a homogeneous quasimorphism $\rho: Ham(\mathbb{A}) \to \RR$ which is 
Hofer-Lipschitz ($|\rho (\phi)| \leq k \|\phi\|$ for all $\phi \in \Ham(\mathbb{A})$) and 
such that $\rho (S_0) = 0$, $\rho (\widehat{\Phi}) = \alpha > 0$. Any $\phi_n \in S_n$ decomposes as 
$\phi_n = \widehat{\Phi}^n \circ s$ for some $s \in S_0$. Hence 
$$|\rho (\phi_n) - \alpha n| = |\rho (\widehat{\Phi}^n \circ s) - n \rho(\widehat{\Phi}) - 0| = 
|\rho (\widehat{\Phi}^n \circ s) - \rho(\widehat{\Phi}^n) - \rho(s)| < R.$$
($R$ denotes the defect of $\rho$). It follows that 
\[
	\rho (\phi_n) = \alpha \cdot n + \delta_{\phi_n} 
\]
where $|\delta_{\phi_n}| \leq R$. Note that $\phi_n^k \in S_{nk}$, hence by homogenuity
\[
	\rho (\phi_n) = \frac{\rho (\phi_n^k)}{k} = \frac{\alpha \cdot n k + \delta_{\phi^k_n}}{k} 
\]
which implies $\rho (\phi_n) = \alpha \cdot n$ in the limit as $k \to \infty$.
This way, $\rho$ extends (up to a rescaling by $\alpha$) the translation number $\tau_\mathcal{D}$.
The Lipschitz property of $\rho$ implies the desired lower bound: 
\begin{equation} \label{eq:l_bound}
	\| \phi_n \| \geq \frac{|\rho (\phi_n)|}{k} = \ |n| \cdot \frac{\alpha}{k}.
\end{equation}

\medskip

In order to build $\rho$ we use the Calabi quasimorphism 
on $Ham (S^2)$ which was constructed by M. Entov and L. Polterovich 
in ~\cite{En-Po:calqm}. We give a brief recollection of the relevant facts.

Let $D$ be an open disk equipped with a symplectic form $\omega$. 
Let $F_t : D \to \RR$, $t \in [0, 1]$ be a time-dependent smooth function with compact support. We define
$\widetilde{\Cal} (F_t) = \int_0^1 \left( \int_D F_t \omega \right) \dd t$. As $\omega$ is exact on $D$,
$\widetilde{\Cal}$ descends to a homomorphism $\Cal_D: \Ham_c(D) \to \RR$ which is called 
the Calabi homomorphism. Clearly, for $\phi \in \Ham_c(D)$, $| \Cal_D(\phi) | \leq Area(D) \cdot \| \phi \|$.

Let $S^2$ be a sphere equipped with a symplectic form $\omega$. Let $Area (S^2) = 2 A$.
For a smooth function $F: S^2 \to \RR$ the Reeb graph $T_F$ is defined as the set of connected
components of level sets of $F$ (for a more detailed definition we refer the reader to ~\cite{En-Po:calqm}). 
For a generic Morse function $F$ this set, equipped with the topology 
induced by the projection $\pi_F: S^2 \to T_F$, is homeomorphic to a tree. 
We endow $T_F$ with a measure given by $\mu (A) = \int_{\pi_F^{-1}(A)} \omega$ 
for any $X \subseteq T_F$ with measurable $\pi_F^{-1}(X)$. $x \in T_F$ is the median of $T_F$
if the measure of each connected component of $T_F \setminus \{x\}$ does not exceed $A$.
This construction can be extended to functions $F$ such that $F \big|_{supp (F)}$ is Morse.

~\cite{En-Po:calqm} describes construction of a homogeneous quasimorphism $\Cal_{S^2} : \Ham (S^2) \to \RR$.
It has the following properties: $\Cal_{S^2}$ is Hofer-Lipschitz ($|\Cal_{S^2}(\phi)| \leq 2A \cdot \| \phi \|$).
In the case when $\phi \in \Ham(S^2)$ is supported in a disk $D$ which is displaceable 
in $S^2$, $\Cal_{S^2} (\phi) = Cal_D (\phi \big|_D)$.
Moreover, for $\phi \in \Ham(S^2)$ generated by an autonomous function $F: S^2 \to \RR$, 
$\Cal_{S^2} (\phi)$ can be computed in the following way. 
Let $x$ be the median of $T_F$ and $X = \pi_F^{-1} (x)$ be the corresponding subset of $S^2$.
Then 
\[
	\Cal_{S^2}(\phi) = \int_{S^2} F \omega - 2A \cdot F(X).
\]

Given a symplectic embedding $j : \mathbb{A} \to S^2$ into a sphere of area $2 A$, 
consider the pullback $Cal_{j} = j^* (Cal_{S^2}) : Ham(\mathbb{A}) \to \RR$. 
Namely, given $\phi \in Ham(\mathbb{A})$, extend $j_* (\phi)$ to $\tilde{\phi} \in \Ham (S^2)$ by identity 
on the complement of $j (\mathbb{A})$. Then $Cal_j (\phi) = \Cal_{S^2} (\tilde{\phi})$.
Clearly, $Cal_{j}$ is a homogeneous quasimorphism.
It has the following properties:
\begin{itemize}
	\item
		$Cal_j (\phi) = \Cal_D(\phi \big|_D)$ for any $\phi$ supported in a disk $D$ of area $A$.
		To see that note that the corresponding $\tilde{\phi} \in Ham(S^2)$ 
		is supported in a displaceable disk $j(D)$ in $S^2$.
	\item
		\[
			\left|\Cal_j(\phi)\right| = \left|\Cal_{S^2}\left(\tilde{\phi}\right)\right| \leq 
			2A \cdot \left\| \tilde{\phi} \right\|_{S^2} \leq 2A \cdot \left\| \phi \right\|_{\mathbb{A}}.
		\]
	\item for an autonomous $\phi$ generated by a compactly supported function $H : \mathbb{A} \to \RR$, 
		\[
			Cal_j (\phi) = \int_\mathbb{A} H \omega - 2A \cdot H(X)
		\]
		where $X \subseteq \mathbb{A}$ is the level set component which is sent by $j$ to the median set
		of $j_* (H)$ in $S^2$.
\end{itemize}

\medskip

Consider the embeddings $j_s : \mathbb{A} \to S^2$ ($0 \leq s \leq 2 A - 1$) into a sphere of area $2 A$ 
($A = Area (\mathcal{D})$) 
that are given by gluing a disk of area $s$ to $S^1 \times \{0\}$ and a disk of area
$(2 A - 1 - s)$ to $S^1 \times \{1\}$.
This construction ensures that $j_s (L)$ bisects $S^2$ 
into two displaceable disks.

We pick $0 \leq s_1 < s_2 \leq 2 A - 1$ and set 
\begin{equation}\label{eq:qm_def}
	\rho = \rho_{s_1, s_2} = Cal_{j_{s_2}} - Cal_{j_{s_1}}.
\end{equation}
Obviously, $\rho$ is a homogeneous quasimorphism on $Ham(\mathbb{A})$ which satisfies 
the Lipschitz property:
\[
	|\rho (\phi)| \leq |Cal_{j_{s_2}} (\phi)| + |Cal_{j_{s_1}} (\phi)| \leq 4 A \cdot \|\phi\|.
\]
Consider the function $\widehat{H}$ which was used to define the Hamiltonian $\widehat{\Phi}$ described 
above. It is easy to see that the ``median'' level set $X_s$ of $\widehat{H}$
which is relevant for the computation of $Cal_{j_s} (\widehat{\Phi})$ is $S^1 \times \{A-s\}$.
Hence 
\[
	Cal_{j_s} (\widehat{\Phi}) = \int_\mathbb{A} \widehat{H} \omega - 2A \cdot \widehat{H}(X_s) = 
		\int_\mathbb{A} \widehat{H} \omega - 2A \cdot (A-s).
\]
This implies
\[
	\rho(\widehat{\Phi}) = Cal_{j_{s_2}} (\widehat{\Phi}) - Cal_{j_{s_1}} (\widehat{\Phi})= 2A \cdot [- (A-s_2) + (A-s_1) ] = 2A \cdot (s_2 - s_1) > 0.
\]
Substitute $s_1 = 0$, $s_2 = 2 A - 1$ into the definition of $\rho$. Using the computation above,
$\alpha = \rho(\widehat{\Phi}) = 2A \cdot (s_2 - s_1) = 2A (2A-1)$. $\rho$ is $4A$-Lipschitz, 
hence for any $\phi_n \in S_n$ ~\eqref{eq:l_bound} implies the lower bound of Theorem~\ref{T:bound}:
\[
	\|\phi_n\| \geq \ |n| \cdot \frac{\alpha}{k} = |n| \cdot \frac{2A (2A-1)}{4A}= |n| \cdot \frac{2A-1}{2}.
\]

\medskip

It is left to show that $\rho$ vanishes on $S_0$. Let 
$S'_0 = \{\phi \in S_0 \, | \, supp(\phi) \subset \mathbb{A} \setminus L\}$
be the subgroup which fixes a neighborhood of $L$ pointwise.

\begin{lem}\label{lm:fix_neigh}
	Let $q$ be a homogeneous quasimorphism which is Hofer-continuous and vanishes on $S'_0$.
	Then $q$ vanishes on $S_0$.
\end{lem}
\begin{proof}
Pick an open disk $D \subset \mathbb{A} \setminus L$. $\Ham_c (D) \subset S'_0$, therefore 
$q$ vanishes on $\Ham_c (D)$. It follows from the results of ~\cite{En-Po-Py:qm-continuity} that 
$q$ is continuous in the $C^0$-topology.

Let $\phi \in S_0$. Applying an appropriate $\psi \in S_0$ with 
arbitrary small Hofer norm, we may ensure that $\psi \circ \phi = \Id$
on $L$. Further, we may find a $C^0$-small diffeomorphism $h \in Ham(\mathbb{A})$ such that
$h \circ \psi \circ \phi = \Id$ in a neighborhood of $L$. 
It follows that $h \circ \psi \circ \phi \in S'_0$ and $q (h \circ \psi \circ \phi) = 0$.
Hofer and $C^0$-continuity of $q$ imply that $q (\phi) = 0$.
\end{proof}

Let $\phi \in S'_0$. We show that $\rho (\phi) = 0$. $\phi$ splits to a composition
$\phi = \phi_\mathcal{D} \circ \phi_P$ where $\phi_\mathcal{D}$ is supported in 
$\mathcal{D}$ and $\phi_P$ is supported in the pair of pants
$P = \mathbb{A} \setminus \mathcal{D}$. $\phi_\mathcal{D}, \phi_P$ have disjoint supports, 
therefore they commute. Hence $\rho (\phi) = \rho(\phi_\mathcal{D}) + \rho (\phi_P)$. 
Note that $\phi_\mathcal{D} \in Ham_c (\mathcal{D})$ and 
$Area(\mathcal{D}) = A$, so $Cal_{j_s} (\phi_\mathcal{D}) = Cal_\mathcal{D} (\phi_\mathcal{D})$
for all $s$. Therefore $\rho (\phi_\mathcal{D}) = 0$.

$\phi_P$ is Hamiltonian on $\mathbb{A}$, but after the restriction to $P$
we have just $\bar{\phi}_P = \phi_P\big|_P \in Symp_c(P)$. In the argument below we apply a sequence of 
deformations to $\phi_P$ in order to get $\phi'_P$ whose restriction $\bar{\phi}'_P \in Ham_c (P)$. All deformations
involved in the process preserve the value $\rho (\phi_P)$. Finally, we show that
$\rho (\phi'_P) = 0$ by explicit computation.
 
The mapping class group 
$\pi_0 (Symp_c (P))$ is isomorphic to $\ZZ^3$ and is generated by Dehn twists near the
three boundary components. For the proof of this fact we refer the reader to 
~\cite{F-M:mapping-class} where the authors show that $\pi_0 (Diff_c (P)) \simeq \ZZ^3$
and is generated by Dehn twists. Note that $\phi, \psi \in Symp_c (P)$ are isotopic in 
$Symp_c$ if and only if they are isotopic in $Diff_c$. As Dehn twists belong to
$Symp_c (P)$, the statement for $\pi_0 (Symp_c (P))$ follows. 

Denote by $T_1, T_0$ Dehn twists near $S^1 \times \{1\}$, $S^1 \times \{0\}$ 
and by $T_L$ a Dehn twist in 
$P$ near $L = \partial \mathcal{D}$. 
Note that we may find a Hamiltonian $\psi_L$ in $S'_0$ with arbitrary small Hofer norm
whose restriction to $P$ realizes the Dehn twist $T_L$.
$\bar{\phi}_P$ is isotopic in $Symp_c (P)$
to some $T_1^{k_1} T_0^{k_0} T_L^{k_L}$ ($k_i \in \ZZ$). 
If $k_L \neq 0$ we replace the original $\phi \in S'_0$ by $\psi_L^{-k_L} \circ \phi \in S'_0$.
As $\|\psi_L\|$ can be chosen to be arbitrarily small, by continuity of $\rho$ it is enough to show 
the desired statement for the deformed $\phi$.
After the replacement $k_L$ vanishes, hence the modified $\bar{\phi}_P \sim T_1^{k_1} T_0^{k_0}$.
Note that $\bar{\phi}_P$ is induced by a Hamiltonian $\phi \in S$. The definition of
$\tau_{\mathcal{D}}$ implies that $k_1 = \tau_\mathcal{D} (\phi) = -k_0$.
The minus sign appears because the opposite orientation of the boundary components
results in the opposite directions of the corresponding Dehn twists. 
Moreover, as $\phi \in S'_0$, $k_1 = \tau_\mathcal{D} (\phi) = 0$.
Therefore the restriction $\bar{\phi}_P$ belongs to the identity component of $Symp_c (P)$. 

Pick $K : \mathbb{A} \to \RR$ supported in a small neighborhood of $\mathcal{D}$ 
such that $K = 1$ in a neighborhood of the closure $\overline{\mathcal{D}}$. 
Denote by $\chi^t$ the time-$t$ map generated by the Hamiltonian flow of $K$. 
The median set for $K$ which is used to compute $Cal_{j_s} (\chi^1)$ is $K^{-1} (1)$. 
It follows that the value
\[
	Cal_{j_s} (\chi^1) = \int_\mathbb{A} K \omega - 2A \cdot 1
\]
does not depend on $s$, hence $\rho (\chi^1) = 0$. 
By homogenuity and continuity of $\rho$, $\rho(\chi^1) = 0$ 
implies that $\rho(\chi^t) = 0$ for all $t$.

Consider the homomorphism $i_* : H^1_c (P; \RR) \to H^1_c (\mathbb{A}; \RR)$ induced by inclusion $i : P \to \mathbb{A}$. 
Both $\chi^t, \phi_P$ are Hamiltonian in $\mathbb{A}$, hence their fluxes are zero in $H^1_c(\mathbb{A}; \RR)$.
After the restriction to $P$, $flux (\chi^t\big|_P), flux (\bar{\phi}_P)$ belong to the one-dimensional subspace $\ker i_* \subset H^1_c (P; \RR)$.
$flux (\chi^t\big|_P) \neq 0$, therefore one can find an appropriate $t_\phi \in \RR$
such that the restriction $\bar{\phi}'_P = \chi^{t_\phi}\big|_P \circ \bar{\phi}_P$ 
of $\phi'_P = \chi^{t_\phi} \circ \phi_P$ has zero flux in $P$.
Hence, as it is shown in ~\cite{Ba:structure-groups}, $\bar{\phi}'_P \in Ham_c (P)$.

Pick a compactly supported function $F_t : P \times [0, 1] \to \RR$ whose flow generates $\bar{\phi}'_P$.
Denote by $U_s$ the complement of the closed disk $\overline{j_s(\mathcal{D})}$ in $S^2$, it is a displaceable disk. 
$(j_s)_* (\phi'_P) \in Ham(S^2)$ and it is supported in $U_s$, therefore
\[
	Cal_{j_s} (\phi'_P) = Cal_{U_s} ((j_s)_* (\phi'_P)) = 
				\int_0^1 \left( \int_{U_s} (j_s)_* F_t \omega \right) \dd t = \int_0^1 \left( \int_P F_t \omega \right) \dd t
\]
is independent of $s$. Hence $\rho(\phi'_P) = 0$. From $\phi'_P = \chi^{t_\phi} \circ \phi_P$,
we obtain $|\rho (\phi'_P) - \rho (\chi^{t_\phi}) - \rho (\phi_P)| = |\rho (\phi_P)| < R$.
It follows that $\rho$ is bounded on the subgroup $S'_0$. As $\rho$ is homogeneous, it vanishes there.

\begin{rem}
	The extension $\frac{\rho}{\rho(\widehat{\Phi})}$ (where $\rho$ is the quasimorphism constructed above) 
	of the translation number $\tau_\mathcal{D}$ is not unique. 
	Choice of different parameters $s_1, s_2$ in ~\eqref{eq:qm_def} gives rise to 
	different extension quasimorphisms. In particular, 
	their Lipschitz constant varies. 
\end{rem}

\section{Upper bound} \label{S:diff}

Pick $n \in \ZZ$. We construct an explicit deformation $\psi_n \in S_n$ such that 
$$\|\psi_n\| \leq (2 A - 1) \cdot |n| + 1.$$ This implies the right-hand side inequality in
Theorem~\ref{T:bound}.

Consider the Hamiltonian flow described in Figure~\ref{F:swap} on the left. 
It is supported in a small neighborhood of two disks and two paths connecting them.
The flow is generated by an autonomous function $H$ such that
$H = 1$ in the internal rectangle, zero outside and is reasonably smoothed in between.
We ask from ``reasonable smoothing'' that $H$ is approximately linear on each of the two disks. 
Clearly, the time-$t$ map of the corresponding flow $\phi_t$ transfers area $t$ from the right 
disk towards the left one and vice versa. The Hofer length $l (\phi^t) = t$. 
\begin{figure}[!htbp]
\begin{center}
\includegraphics[width=0.8\textwidth]{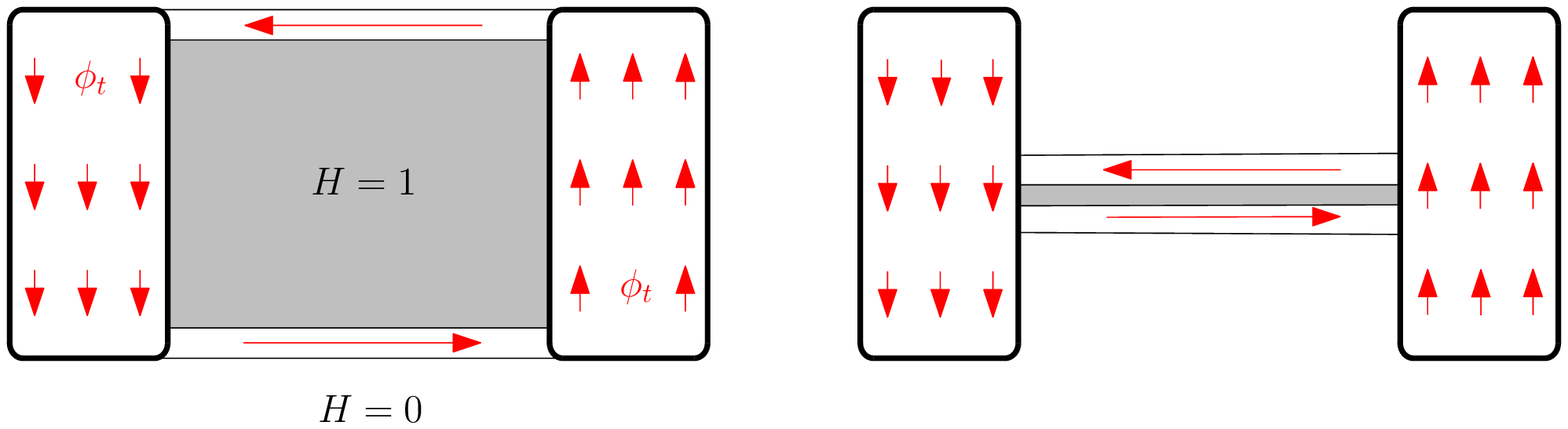}
\caption{}
\label{F:swap}
\end{center}
\end{figure}
We may deform this construction to make the internal rectangle very thin. Then it looks 
as depicted on the right, namely, the flow is supported in a small 
neighborhood of the two disks and a single path connecting them.

We apply this construction on $\mathbb{A}$ as drawn in Figure~\ref{F:spiral}, on the left. 
\begin{figure}[!htbp]
\begin{center}
\includegraphics[width=0.72\textwidth]{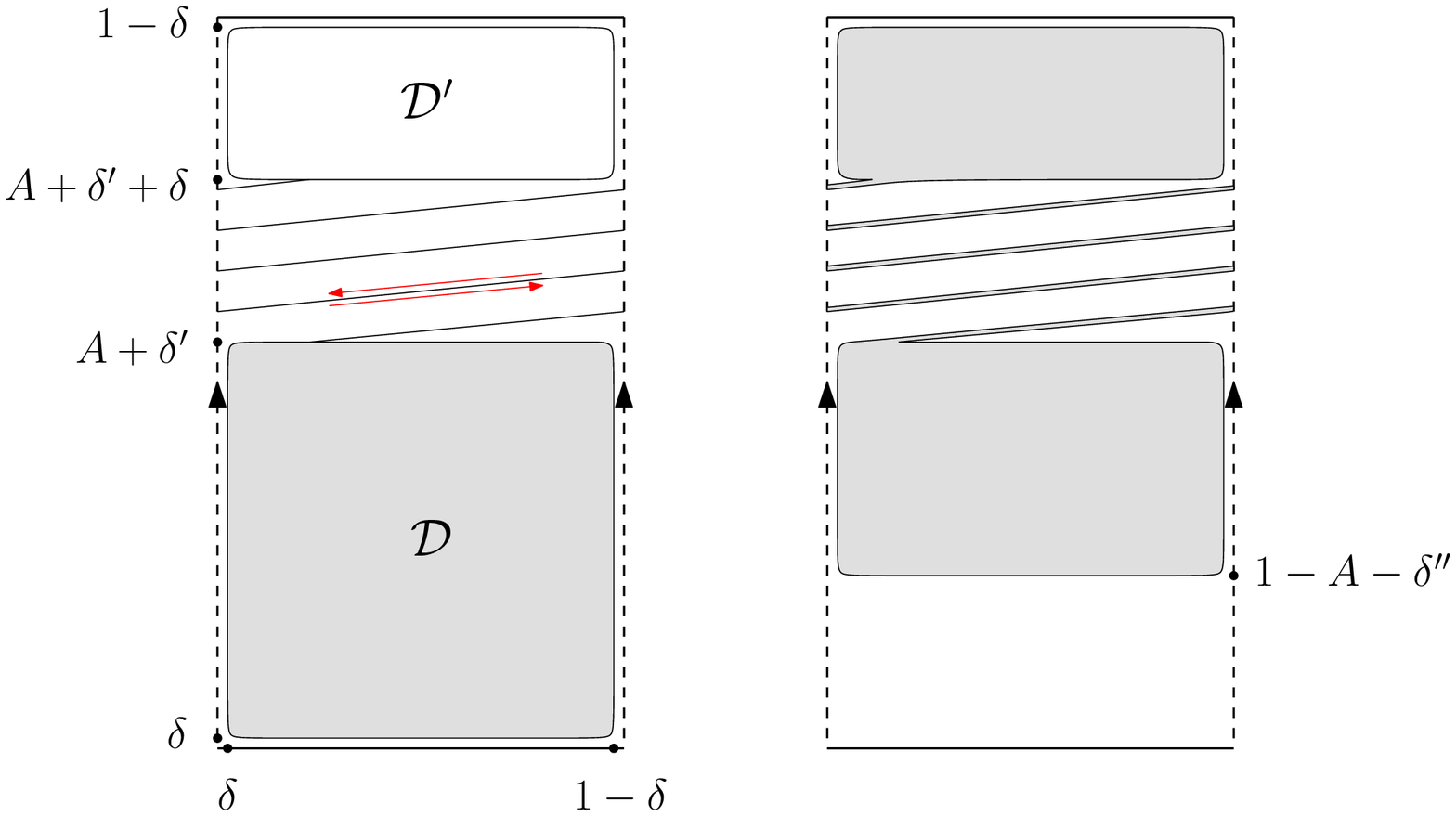}
\caption{}
\label{F:spiral}
\end{center}
\end{figure}
Namely, we change the coordinate system, if necessary, so that $\mathcal{D}$
is approximated by rectangle 
$(\delta, 1-\delta) \times (\delta, A + \delta') \subset S^1 \times (0, 1)$.
Here $A = Area(\mathcal{D})$ and $\delta, \delta' << 1$. Denote by $\mathcal{D}'$ a disk
given by smoothing the rectangle $(\delta, 1-\delta) \times (A + \delta' + \delta, 1-\delta) \subset \mathbb{A} \setminus \mathcal{D}$.
Note that $Area(\mathcal{D}') = 1 - A - \varepsilon$ where $\varepsilon$ is comparable to 
$\delta + \delta'$. Connect $\mathcal{D}$ with $\mathcal{D}'$ by a spiral which lies in 
$S^1 \times [A + \delta', A + \delta' + \delta]$ and makes
$n$ loops around the $S^1$ coordinate. Now apply the flow as in Figure~\ref{F:swap} (right) 
to transfer a portion of area from $\mathcal{D}$
to $\mathcal{D}'$ using a small tubular neighborhood of the spiral. We stop the process
when entire $\mathcal{D}'$ is covered by the deformed $\mathcal{D}$ 
(see Figure~\ref{F:spiral}, on the right). We may arrange this construction so that the deformed $\mathcal{D}$
lies above $S^1 \times \{1-A-\delta''\}$ where $\delta''$ is comparable to $\delta + \delta'$. 
The length of this deformation is 
$Area(\mathcal{D}') + \varepsilon' = 1 - A - \varepsilon + \varepsilon'$.

In the next step we rotate horizontal circles $S^1 \times \{h\}$. The rotation angle is chosen 
in a way that the annulus $S^1 \times (1-A-\delta'', A + \delta')$ 
(which contains the intersection of the deformed $\mathcal{D}$ with the original $\mathcal{D}$) is rotated $n$ times around the $S^1$ coordinate, 
and the angle gradually decreases to zero along the annulus containing the spiral ($S^1 \times (A + \delta', A + \delta' + \delta)$).
As the result, the spiral is unfolded to a vertical line. 
The remaining part of $\mathbb{A}$ is remains fixed. 
Denote by $H$ an autonomous Hamiltonian which generates this flow, its graph is shown schematically in
Figure ~\ref{F:finish} (left). $H$ is cut off in the $\delta$-neighborhood
of $\partial \mathbb{A}$ to make it compactly-supported and is smoothed near the singular points.
The cutoff will result in a strong flow near $S^1 \times \{0, 1\}$, but this has no effect 
on our construction as this flow stays away from our area of interest.
Energy needed for this deformation is 
\[
	|n| \cdot Area (S^1 \times (1-A-\delta'', A + \delta')) + \varepsilon'' = |n| \cdot (2A-1+\delta'+ \delta'') + \varepsilon''
\]
where $\varepsilon''$ covers energy consumed by rotation of the spiral area and smoothing costs. $\varepsilon''$ is bounded by $|n| \cdot \delta$.

Finally we are in the situation shown in Figure~\ref{F:finish} (right) where all points of the 
deformed $\mathcal{D}$ are already rotated $n$ times around the $S^1$ coordinate.
\begin{figure}[!htbp]
\begin{center}
\includegraphics[width=0.55\textwidth]{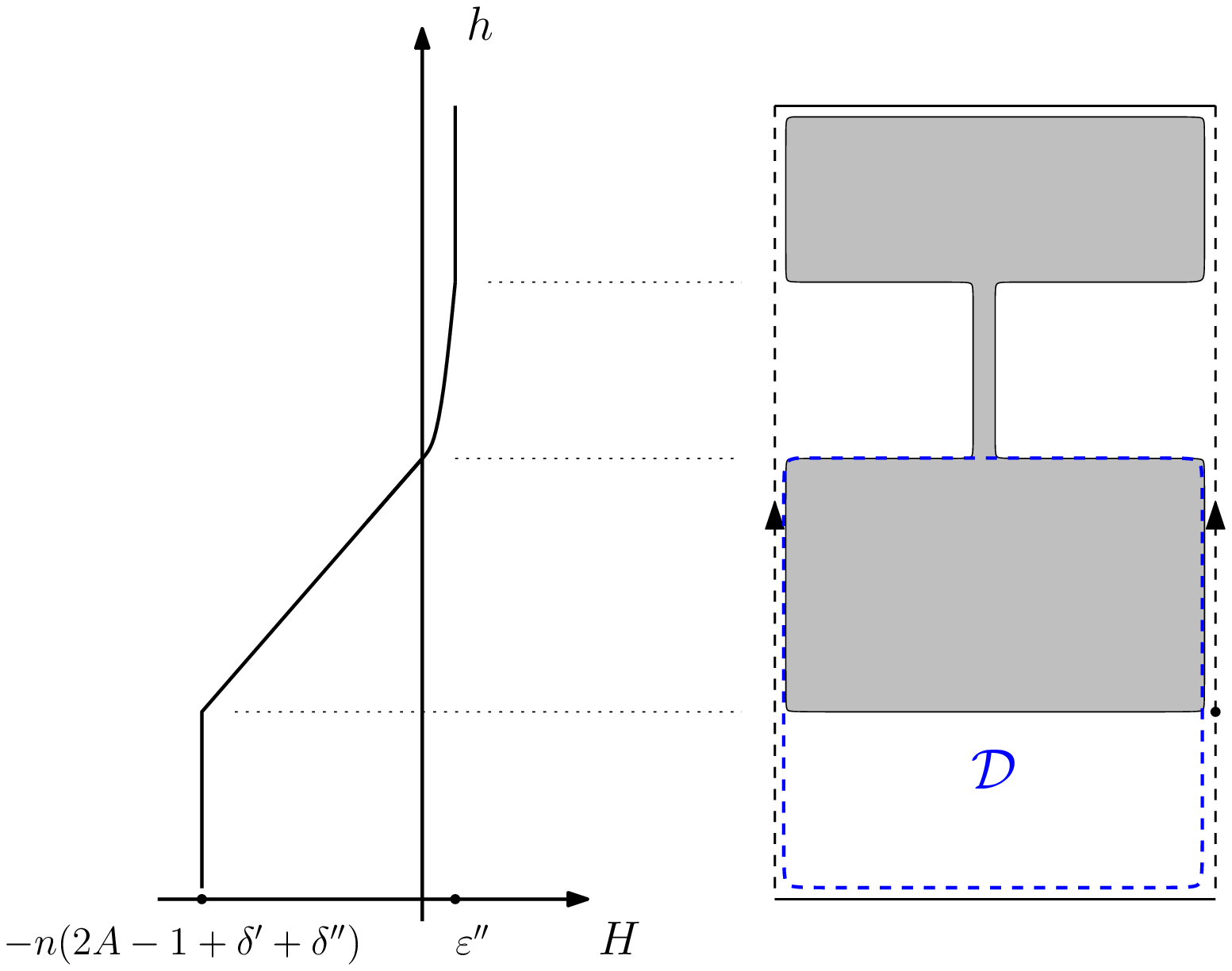}
\caption{}
\label{F:finish}
\end{center}
\end{figure}
We move the deformed disk in an obvious way down in order to fill $\mathcal{D}$. The length of such deformation
is bounded by $Area(\mathbb{A}) - Area (\mathcal{D}) = 1-A$.

Summarizing the argument above, 
we have constructed a Hamiltonian isotopy connecting $\Id$ to $S_n$ whose 
Hofer length is bounded by
\[
	(1 - A - \varepsilon + \varepsilon') + |n| \cdot (2A - 1+ \delta'+ \delta'') + \varepsilon'' + (1-A) = 
	|n| \cdot (2A - 1) + (2-2A) + \hat{\varepsilon} < |n| \cdot (2A - 1) + 1
\] 
for appropriate choices of parameters $\delta$ and $\varepsilon$.

\bibliography{bibliography}

\end{document}